\documentclass[11pt, a4paper]{amsart}
\usepackage{amssymb, array, amsmath, amscd, pdfpages, enumerate, amsthm, fixltx2e, setspace, hyperref}

\DeclareMathOperator*{\Ran}{Ran}
\DeclareMathOperator*{\Ker}{Ker}
\DeclareMathOperator*{\Fix}{Fix}
\DeclareMathOperator*{\dist}{dist}

\newcommand{\ii}{\mathrm{i}}
\newcommand{\dd}{\mathrm{d}}

\newcommand{\B}{\mathcal{B}}

\newcommand{\T}{\mathbb{T}}

\newcommand{\CC}{\mathbb{C}}

\newcommand{\ZZ}{\mathbb{Z}}

\newcommand\ps[2]{\left\langle #1,#2\right\rangle}

\newtheorem{thm}{Theorem}[section]
\newtheorem{prp}[thm]{Proposition}
\newtheorem{lem}[thm]{Lemma}
\newtheorem{cor}[thm]{Corollary}
\theoremstyle{definition}
\newtheorem{rem}[thm]{Remark}

\numberwithin{equation}{section}

\begin{document}

\title[Ritt operators and  the method of alternating projections]{Ritt operators and convergence in the method of alternating projections}

\author{Catalin Badea}
\address{Universit\'e Lille 1, Laboratoire Paul Painlev\'e, CNRS UMR 8524, 59655 Villeneuve d'Ascq, France}
\email{badea@math.univ-lille1.fr}

\author{David Seifert}
\address{St John's College, St Giles, Oxford\;\;OX1 3JP, United Kingdom}
\email{david.seifert@sjc.ox.ac.uk}

\begin{abstract}
Given $N\ge2$ closed subspaces $M_1,\dotsc, M_N$ of a Hilbert space $X$, let $P_k$ denote the orthogonal projection onto $M_k$, $1\le k\le N$. It is known that the sequence $(x_n)$, defined recursively by $x_0=x$ and $x_{n+1}=P_N\cdots P_1x_n$ for $n\ge0$, converges in norm to $P_Mx$ as $n\to\infty$ for all $x\in X$, where $P_M$ denotes the orthogonal projection onto $M=M_1\cap\dotsc\cap M_N$.  Moreover, the rate of convergence is either exponentially fast for all $x\in X$ or as slow as one likes for appropriately chosen initial vectors $x\in X$. We give a new estimate in terms of natural geometric quantities on the rate of convergence in the case when it is known to be exponentially fast. More importantly, we then show that even when the rate of convergence is arbitrarily slow there exists, for each real number $\alpha>0$,  a dense subset $X_\alpha$ of $X$ such that $\|x_n-P_Mx\|=o(n^{-\alpha})$ as $n\to\infty$ for all $x\in X_\alpha$. Furthermore, there exists another dense subset $X_\infty$ of $X$ such that, if $x\in X_\infty$, then $\|x_n-P_Mx\|=o(n^{-\alpha})$ as $n\to\infty$ for \emph{all} $\alpha>0$. These latter results are obtained as consequences of general properties of Ritt operators. As a by-product, we also strengthen the unquantified convergence result by showing that $P_M x$ is in fact the limit of a series which converges unconditionally.
 \end{abstract}

\thanks{The authors would like to thank H.\ Bauschke, C.\ Cuny, F.\ Deutsch, C.\ Le Merdy, M.\ Lin, V.\ M\"uller, S.\ Reich and Y.\ Tomilov for helpful comments and discussions. Work on this project was supported in part by the Labex CEMPI (ANR-11-LABX-0007-01).}

\maketitle

\section{Introduction}\label{intro} 

The von Neumann-Halperin method of cyclic alternating projections is a general iterative method for finding the best approximation to any given point in a Hilbert space from the intersection of a finite number of closed subspaces. This method and its extensions and variants have been applied in a variety of fields, including solving linear equations,
linear prediction theory, image restoration, computed tomography; see \cite{De92} for a survey. 

The convergence of the algorithm given by the method of cyclic alternating projections in the case of several subspaces in a Hilbert space is based upon the following result. 
Given $N\ge2$ closed subspaces $M_1,\dotsc, M_N$ of a Hilbert space $X$, let $P_k$ denote the orthogonal projection onto $M_k$, $1\le k\le N$, and let $P_M$ denote the orthogonal projection onto $M=M_1\cap\dotsc\cap M_N$. Consider the sequence $(x_n)$ defined recursively by $x_0=x\in X$ and $x_{n+1}=P_N\cdots P_1x_n$ for $n\ge0$. We thus alternately project onto one subspace, then onto the next one, in cyclic order.
It is known that 
\begin{equation}\label{conv}
\lim_{n\to\infty}\|x_n-P_M x\|=0
\end{equation}
for all initial points $x\in X$. This was first shown for $N=2$ by J.~von Neumann in 1933 and published in \cite{vN49} and then for the general case in \cite{Hal62}; see also \cite{NeSo06} for a simpler proof. 

A question which arises naturally is whether anything can be said about the \emph{rate} of convergence in \eqref{conv}. We say that the convergence is \emph{exponentially fast} if there exist $r\in [0,1)$ and, for each $x\in X$, a constant $C_x>0$ such that
$$\|x_n-P_M x\|\le C_x r^n,\quad n\ge0,$$
and we say that the convergence is \emph{arbitrarily slow} if, given any sequence $(r_n)$ of positive real numbers satisfying $r_n\to0$ as $n\to\infty$, there exists $x\in X$ such that 
$$\|x_n-P_M x\|\ge r_n,\quad n\ge0.$$
By the uniform boundedness principle,  exponentially fast convergence is equivalent to the existence of constants $C>0$ and $r\in[0,1)$ such that  $\|T^n-P_M \|\le C r^n$ for all $n\ge0,$ where $T=P_N\cdots P_1$. In what follows, we shall also use the notions of exponentially fast and arbitrarily slow convergence more generally in cases such as the above, where the powers of an operator converge strongly to a limit.

It was shown in \cite[Theorem~1.4]{BaDeHu09} for the case $N=2$ and in \cite[Theorem~6.4]{DeHu10} for the general case that the convergence in \eqref{conv} is  exponentially fast if and only 
if the subspace $M_1^\perp+\dots+ M_N^\perp$ is closed in $X$, and that otherwise the convergence is arbitrarily slow. 
This dichotomy result, obtained independently in \cite{BGM10, BGM11}, was extended in a number of directions in \cite{BGM11}, for instance by giving characterisations of exponentially fast convergence in terms of the geometric relationship between the subspaces $M_1,\dotsc M_N$. Other relevant sources include \cite{BaBoLe97,DeHu10,DeHuSurvey1, DeHu15, XuZik} and \cite[Chapter~9]{De01}.

Suppose that $M_1^\perp+\dots+ M_N^\perp$ is not closed in $X$ (for reasons which will become clear later on, we shall say in this case that the subspaces $M_1, \cdots , M_N$ are \emph{aligned}). Therefore the convergence of $(x_n)$ can be slow for \emph{some} initial points $x\in X$. For applications in approximation theory for instance it is vital to have some estimate on the rate of convergence in \eqref{conv} for \emph{particular} initial vectors $x\in X$; see for instance the discussion in \cite[Section~4]{PuReZa13}. It is also natural to ask, in the case when the subspaces are aligned, where the initial points lie for which the convergence is arbitrarily slow. It was conjectured in \cite[Remark~6.5(2)]{DeHu10} that in general they must lie in the complement  of $M\oplus(M_1^\perp+\cdots +M_N^\perp)$, which is a dense subspace of $X$. The principal purpose of this paper is to prove that even when the convergence in  \eqref{conv} is arbitrarily slow there  exists, for each  $\alpha>0$, a dense subspace $X_\alpha$ of $X$ such that 
\begin{equation}\label{poly}
\|x_n-P_Mx\|=o(n^{-\alpha}),\quad n\to\infty,
\end{equation}
 for all $x\in X_\alpha$. We also show that there exists a further dense subspace $X_\infty$ of $X$ such that for all initial vectors $x\in X_\infty$ the statement in \eqref{poly}  holds for \emph{all} $\alpha>0$.  Although we give a fairly explicit description of these subspaces, the question raised in \cite[Remark~6.5(2)]{DeHu10} remains open.
 
 Our results, presented in Theorem~\ref{Proj}, are obtained by combining ideas contained in \cite{BGM11} with  general results in operator theory arising from works such as \cite{Cro08, KaPo08, LaLeMe13}, including in particular a characterisation of so-called unconditional Ritt operators in terms of the numerical range (field of values); see Theorem~\ref{Stolz}.  As a by-product, we also strengthen the unquantified  statement \eqref{conv} itself by showing that the limit is in fact the limit of a series which converges unconditionally. 
 
 {\bf Organisation of the paper.} 
In Section~\ref{sec:general} we present some results on (unconditional) Ritt operators, together with a review of Ritt operators. Then in Section~\ref{sec:inclination} we introduce a variant of a geometric quantity, called the inner $\ell^2$-inclination, which is related to the Friedrichs number from \cite{BGM10, BGM11}. We use this geometric quantity to estimate the rate of convergence in the method of alternating projections. This implies an inequality for $\|T-P_M\|$ in terms of some geometric quantities which, along with known results on Ritt operators, will be used in Section~\ref{sec:proj} to prove the main results of the paper. 
  
 {\bf Notation.} 
 The notation used is standard throughout. In particular, we write $X^*$ for the dual space of a complex Banach space $X$ and  $\B(X)$ for the algebra of bounded linear operators on $X$. In what follows, all Banach spaces will implicitly be assumed to be complex. An operator $T\in \B(X)$ will be said to be \emph{power-bounded} if $\sup_{n\ge0}\|T^n\|<\infty$. Further, we write $\Ran(T)$ for the range of $T$, $\Ker(T)$ for its kernel and we let $\Fix(T)=\Ker(I-T)$ denote the set of fixed points of $T$. If $X$ is a Hilbert space with inner product $\ps{\cdot}{\cdot}$ the adjoint operator of $T\in\B(X)$ will be denoted by $T^*$. We write $\sigma(T)$ for the spectrum of $T$, $r(T)$ for its spectral radius and, given $\lambda\in\CC\backslash\sigma(T)$, we let $R(\lambda,T)$ denote the resolvent operator $(\lambda-T)^{-1}$.  Given two sequences $(x_n)$, $(y_n)$ of non-negative real numbers, we write $x_n=O(y_n)$ as $n\to\infty$ if there exists a constant $C>0$ such that $x_n\le Cy_n$ for all sufficiently large $n\ge0$. If $y_n>0$ for all sufficiently large $n\ge0$, we write $x_n=o(y_n)$ as $n\to\infty$ if $x_n/y_n\to0$ as $n\to\infty$. We write $\T$ for the unit circle $\{\lambda\in\CC:|\lambda|=1\}$.

\section{Ritt operators and unconditional Ritt operators}\label{sec:general} 

The aim of this section, partly expository, is to introduce the notions of Ritt operators and unconditional Ritt operators and to present some results about them. These results, some of which may be known to specialists in operator theory, will be crucial in dealing with products of projections in Section~\ref{sec:proj}. 

An operator $T$ is said to be a \emph{Ritt operator} if $r(T)\le1$ and there exists a constant $C>0$ such that 
\begin{equation}\label{eq:Ritt}
\|R(\lambda,T)\|\le\frac{C}{|\lambda-1|},\quad |\lambda|>1.
\end{equation} 
This resolvent estimate, stronger than the Kreiss condition, took its name from \cite{ritt}. It can be proved that any Ritt operator $T$ is power-bounded and satisfies $\|T^n(I-T)\|=O(n^{-1})$ as $n\to\infty$, and the converse is also true; see for instance \cite{NaZe99} and \cite{Ly99}. This estimate shows that Ritt operators are, in a certain sense, discrete analogues of analytic semigroups; see \cite{CouSaCo90} where operators satisfying $\|T^n(I-T)\|=O(n^{-1})$ as $n\to\infty$ are studied and called analytic operators.
Moreover, \eqref{eq:Ritt} implies that $\sigma(T)\cap\T\subset\{1\}$, and it is clear that if $T$ is a Ritt operator then so is any operator which is \emph{similar} to $T$, that is to say any operator of the form $Q^{-1}TQ$ for some invertible operator $Q\in\B(X)$. Recall also that for any power-bounded operator $T$, the operator $I-T$ is \emph{sectorial} of angle (of most) $\pi/2$, so that the \emph{fractional powers} $(I-T)^\alpha$ are well-defined for all $\alpha\ge0$, either as a contour integral
$$(I-T)^\alpha=\frac{1}{2\pi\ii}\oint_\Gamma \lambda^\alpha R(\lambda,I-T)\,\dd\lambda,$$
where $\Gamma$ is any sufficiently smooth contour that contains the origin and otherwise encloses $\sigma(I-T)$ without touching it, or alternatively as the absolutely convergent series 
$$(I-T)^\alpha=\sum_{n=0}^\infty(-1)^n\binom{\alpha}{n}T^n.$$ 
See for instance \cite{HaTo10} for details on sectorial operators and fractional powers. 

An operator $T$ on a Banach space $X$ is said to be \emph{mean ergodic} if the averages $\frac{1}{n+1}\sum_{k=0}^nT^kx$, $n\ge0$, converge in norm to a limit as $n\to\infty$ for each $x\in X$. Details on mean ergodic operators may be found in \cite[Section~2.1]{Kre85}. The next result characterises Ritt operators among all power-bounded mean ergodic operators on a Banach space in terms of decay of orbits; see \cite[Remark~3.12]{Se2} for a  related result.
   
\begin{thm}\label{Ritt_orbits}
Let $X$ be a  Banach space and suppose that $T\in\B(X)$ is power-bounded and mean-ergodic. Then $T$ is a Ritt operator if and only if, for each $\alpha>0$, $\|T^nx\|=o(n^{-\alpha})$ as $n\to\infty$ for all $x\in \Ran(I-T)^\alpha$.
\end{thm}

\begin{proof}
By the uniform boundedness principle, the condition for $\alpha=1$ implies that $\|T^n(I-T)\|=O(n^{-1})$ as $n\to\infty$. Since $T$ is assumed to be power-bounded, it follows that $T$ is a Ritt operator. The converse is proved in  \cite[Corollary~6.2]{CoCuLi14} and \cite[Proposition~2]{CouSaCo90}.
\end{proof}

Following \cite{KaPo08}, an operator $T$ on a Banach space $X$ is said to satisfy the \emph{unconditional Ritt condition}, or to be an \emph{unconditional Ritt operator},  if there exists a constant $C>0$ such that, for any $n\ge0$ and any $a_0,\dotsc,a_n\in\CC$,
\begin{equation}\label{uncond_Ritt}
\bigg\|\sum_{k=0}^na_kT^k(I-T)\bigg\|\le C\max_{0\le k\le n}|a_k|.
\end{equation}
Note that if $T$ satisfies the unconditional Ritt condition then so does any operator which is similar to $T$. Furthermore, it is straightforward to show that \eqref{uncond_Ritt}
 is equivalent to having
\begin{equation}\label{wuC}
\sum_{n=0}^\infty|\phi(T^n(I-T)x)|\le C\|x\|\|\phi\|
\end{equation}
 for all $x\in X$, $\phi\in X^*$, and it is shown in \cite[Proposition~4.3]{KaPo08} that if $T$ satisfies the unconditional Ritt condition \eqref{uncond_Ritt},  then \eqref{eq:Ritt} holds, so $T$ is a Ritt operator. Theorem~\ref{Stolz} below provides a characterisation of unconditional Ritt operators acting on a Hilbert space. Recall that the \emph{numerical range} (or \emph{field of values}) $W(T)$ of a bounded linear operator $T$ on a Hilbert space $X$ is defined as $W(T)=\{\ps{Tx}{x}:x\in X\;\mbox{with}\;\|x\|=1\}$. 
 Furthermore, given  $\theta\in[0,\pi/2)$, we write $S_\theta$ for the \emph{Stolz domain} with half-angle $\theta$ at 1, that is to say the convex hull of the set $\{\lambda\in\CC:|\lambda|\le\sin\theta\}\cup\{1\}$.  Note that some authors reserve the term `Stolz domain' for slightly different regions.

\begin{thm}\label{Stolz}
Let $X$ be a Hilbert space and let $T\in \B(X)$. Then $T$ satisfies the  unconditional Ritt condition if and only if $T$ is similar to an operator whose numerical range is contained in a Stolz domain.
\end{thm}

\begin{proof}
Suppose first that there exists an invertible operator $Q\in\B(X)$ such that the operator  $S=Q^{-1}TQ$ satisfies $W(S)\subset S_\theta$ for some $\theta\in[0,\pi/2).$ Let $n\ge0$ and $a_1,\dotsc,a_n\in\CC$ be given, and consider the polynomial $p_n$ defined by
$$p_n(\lambda)=\sum_{k=0}^n a_k\lambda^k(1-\lambda),\quad\lambda\in\CC.$$
Then $p_n(1)=0$ and, for $|\lambda|<1$,
$$|p_n(\lambda)|\le \max_{0\le k\le n}|a_k|\,\sum_{m\ge0}|\lambda|^m|1-\lambda|= \frac{|1-\lambda|}{1-|\lambda|}\,\max_{0\le k\le n}|a_k|.$$
Since $W(S)\subset S_{\theta}$, a simple estimate shows that 
$$\sup_{\lambda\in W(S)}|p_n(\lambda)|\le C\max_{0\le k\le n}|a_k|$$
for some $C>0$ which is independent of $n\ge0$ and $a_1,\dotsc,a_n\in\CC$. Hence 
$$\|p_n(S)\|\le 12C\max_{0\le k\le n}|a_k|,$$
by \cite[Theorem~2]{Cro07}, so $S$ satisfies  \eqref{uncond_Ritt}. Since $T$ is similar to $S$, $T$ itself satisfies the unconditional Ritt condition.

Now suppose conversely that $T$ satisfies the unconditional Ritt condition. Since the notions of `boundedness' and `$R$-boundedness' for families of operators coincide in the Hilbert space setting (see for instance \cite[Section~1.9]{KuWe04}), it follows from \cite[Theorem~4.2]{LaLeMe13} that $T$ admits a bounded $H^\infty$ functional calculus on $S_\theta$ for some $\theta\in[0,\pi/2)$, which implies that there exists a constant $K\ge1$ such that 
$$\|q(T)\|\le K\sup_{\lambda\in S_\theta}|q(\lambda)|$$
for all rational functions $q$ with poles outside $S_\theta$.  Hence  $S_\theta$ is a $K$-spectral set for $T$; see for instance \cite{BaBe14} for details on (complete) $K$-spectral sets. Let  $\theta'\in(\theta,\pi/2)$. By \cite[Theorem~6.1]{FraMcI98} the set $S_{\theta'}$ is a complete $K$-spectral set for $T$, and it follows from \cite{Pau84} that there exists an invertible operator $Q\in\B(X)$ such that $S_{\theta'}$ is a complete 1-spectral set for $Q^{-1}TQ$. Since Stolz domains are convex, $S_{\theta'}$ coincides with the intersection of all half-planes which are 1-spectral for $Q^{-1}TQ$.  It follows easily from \cite[Section~5.3]{vNeu51} that any such half-plane contains the numerical range of $Q^{-1}TQ$, and hence $W(Q^{-1}TQ)\subset S_{\theta'}$.
\end{proof}

\begin{rem}\begin{enumerate}[(a)]
\item The idea used in the first part of the above proof goes back to  \cite[Lemma 5.2]{Cro08}, \cite{DeDe99} and \cite[Lemma~4.1]{LaLeMe13}. Furthermore, noting that any operator on a Hilbert space admitting a 1-spectral set contained in the closed unit disc is a contraction, the above proof yields further equivalent conditions in Theorem~\ref{Stolz}, namely that $T$ is similar to a contraction admitting a Stolz domain as a complete $1$-spectral set, that $T$ admits a Stolz domain as a (complete) $K$-spectral set for some $K\ge1$, and  that $T$ admits a bounded $H^\infty$ functional calculus on a Stolz domain. From the latter it then follows by \cite[Theorem~8.1]{LeMe14} that these conditions are also equivalent to $T$ being similar to a contractive Ritt operator, and to $T$ being power-bounded and satisfying certain square-function estimates; see also  \cite{CoCuLi14, KaPo08}. In particular, a contraction on a Hilbert space is a Ritt operator if and only if it is an unconditional Ritt operator. Several extensions to more general Banach spaces can be found for instance in \cite{LeMe14}. We choose not to list these equivalent conditions in Theorem~\ref{Stolz} since they will not be used in what follows.
\item An alternative way to conclude the proof of Theorem \ref{Stolz}, having established that $S_{\theta'}$ is a complete 1-spectral set for $Q^{-1}TQ$, is as follows.  By  \cite[Section~1.2]{Arv72} there exists a normal operator $S$ on a Hilbert space containing $X$ such that $\sigma(S)$ is contained in the boundary $\partial S_{\theta'}$ of the Stolz domain; see also \cite[Chapter~4]{Pau02}. Moreover, the closure of $W(Q^{-1}TQ)$ is contained in the closure of $W(S)$. Since for a normal operator the closure of the numerical range coincides with the convex hull of the spectrum (see for instance  \cite[Theorem~1.4-4]{GuDu97}), it follows that $W(Q^{-1}TQ)\subset S_{\theta'}$. This gives the additional equivalent statement that $T$ is similar to a contraction admitting a normal dilation whose spectrum is contained in the boundary of a Stolz domain. A direct construction of the normal dilation has been given in \cite{CuLi15}. For further related results see also \cite{deLau98}.
\end{enumerate}
\end{rem}

We now come to the main result of this section, Theorem~\ref{thm:gen} below. This requires a little preparation. Let $c_0$ denote the space of sequences $(a_n)$ of complex numbers such that $|a_n|\to0$ as $n\to\infty$, endowed with the supremum norm. We say that a Banach space $X$ \emph{contains a copy of $c_0$} if there exists a closed subspace of $X$ which is isomorphic to $c_0$. Recall also that, given a sequence $(x_n)$ in $X$ and given $x\in X$, the formal series $\sum_{n\ge0} x_n$ is said to be \emph{weakly unconditionally Cauchy} if, for every $\phi\in X^*$, 
$\sum_{n\ge0}|\phi(x_n)|<\infty.$
Furthermore, the series $\sum_{n\ge0} x_n$ is said to \emph{converge unconditionally to $x$} if the series $\sum_{n\ge0} x_{\pi(n)}$ converges to $x$ in the usual sense for all permutations $\pi$ of $\ZZ_+$. The series $\sum_{n\ge0} x_n$ is said to be \emph{unconditionally convergent} if it converges unconditionally to some $x\in X$. The following theorem, then, is a general dichotomy result, which will be applied in the form of Corollary~\ref{Hilbert_cor} to iterated projections in Section~\ref{sec:proj}. Given a  sequence $(r_n)$  of non-negative numbers such that $r_n\to0$ as $n\to\infty$, we say that the convergence is \emph{super-polynomially fast} if $r_n=o(n^{-\alpha})$ as $n\to\infty$ for all $\alpha>0$.

\begin{thm}\label{thm:gen}
Let $X$ be a Banach space which does not contain a copy of $c_0$, and suppose that $T\in \B(X)$ satisfies the unconditional Ritt condition.  Then the space $X$ splits as $X= \Fix(T)\oplus Z$, where $Z$ is the closure of $\Ran(I-T)$, and the series $\sum_{n\ge0}T^n(I-T)x$ converges unconditionally to $x-Px$ for all $x\in X$, where $P$ is the bounded projection onto $\Fix(T)$ along $Z$. In particular,
\begin{equation}\label{conv_gen}
\lim_{ n\to\infty}\|T^nx-Px\|=0
\end{equation}
for all $x\in X$. Moreover, if $\Ran(I-T)$ is closed then the convergence is exponentially fast for all $x\in X$, while if $\Ran(I-T)$ is not closed the convergence is arbitrarily slow but for each $\alpha>0$ there  exists a dense subspace $X_\alpha$ of $X$ such that 
\begin{equation}\label{rate_gen}
\|T^nx-Px\|=o(n^{-\alpha}),\quad n\to\infty,
\end{equation}
for all $x\in X_\alpha$. Furthermore, there exists a dense subspace $X_\infty$ of $X$ such that for all $x\in X_\infty$ the convergence in \eqref{conv_gen} is super-polynomially fast.
\end{thm}

\begin{proof}

Since $T$ is an unconditional Ritt operator,  \eqref{wuC} holds for all $x\in X$ and all $\phi\in X^*$. Thus the series $\sum_{n\ge0} T^n(I-T)x$ is weakly unconditionally Cauchy for each $x\in X$. Since $X$ contains no copy of $c_0$, it follows from the Bessaga-Pe\l czy\'{n}ski theorem \cite[Theorem~5]{BesPel58} that the series converges unconditionally to some $z\in Z$. Noting that the partial sums of the series form a telescoping sum, it follows immediately that $\|T^nx-x+z\|\to0$ as $n\to\infty$. Next we identify $z\in Z$. Since the powers of $T$ converge strongly to a limit, $T$ is mean ergodic.  It follows from power-boundedness of $T$ and \cite[Chapter~2, Theorem~1.3]{Kre85} that  $X=\Fix(T)\oplus Z$. By virtue of $T$ being a Ritt operator we have that $\|T^n(I-T)\|\to0$  as $n\to\infty$. In particular, $\|T^nx\|\to0$ as $n\to\infty$ for all $x\in Y$, and by the power-boundedness of  $T$  this result extends to all $x\in Z$. It follows that \eqref{conv_gen} holds and hence that  $z=x-Px$, as required.

It remains to prove the quantified statements. Let $Y=\Ran(I-T)$, so that $Z$ is the closure of $Y$, and let $S$ denote the restriction of $T$ to $Z$. Then $\sigma(S)\subset\sigma(T)$ and  the operator $I-S$ maps $Z$ bijectively onto $Y$. It follows from the Inverse Mapping Theorem that $1\in\sigma(S)$ if and only if $Y\ne Z$. Using the fact that $\sigma(T)$ is contained in the closed unit disc with $\sigma(T)\cap\T\subset\{1\}$, we obtain the equivalent statement that $r(S)<1$ if and only if $Y$ is closed. But $\|T^n-P\|=\|S^n\|$ for all $n\ge0$, so if $Y$ is closed then the convergence in \eqref{conv_gen} is exponentially fast. On the other hand, if $Y$ is not closed and consequently $r(S)=1$, then it follows from \cite[Theorem~1]{Mue88}  that, given any sequence $(r_n)$ of positive terms satisfying $r_n\to0$ as $n\to\infty$, there exists $x\in Z$ such that $\|S^nx\|\ge r_n$ for all $n\ge0$. Since $S^nx=T^n x-P_Mx$ for all $n\ge0$,  the convergence in \eqref{conv_gen} is  arbitrarily slow. Given $\alpha>0$, let $X_\alpha=\Fix(T)\oplus \Ran(I-T)^\alpha$. Since $\Ran(I-T)^\alpha$ is dense in $Z$ and $X=\Fix(T)\oplus Z$, $X_\alpha$ is dense in $X$.  For $x\in X_\alpha$, let $y=x-Px.$ Then  $y\in \Ran(I-T)^\alpha$ and $T^nx-Px=T^ny$ for all $n\ge0$. Since $T$ is mean ergodic, \eqref{rate_gen}  follows from Theorem~\ref{Ritt_orbits}. For the final claim, let $X_\infty=\bigcap_{n=1}^\infty X_n$. It follows from the Esterle-Mittag-Leffler theorem \cite[Theorem~2.1]{Est84} that $X_\infty$ is dense in $X$, and it is clear that for initial vectors $x\in X_\infty$ the convergence in \eqref{conv_gen} is super-polynomially fast. This completes the proof.
\end{proof}

\begin{rem} \label{gen_rem}
\begin{enumerate}[(a)]
\item\label{rem:frac} As the above proof shows, the dense subspaces $X_\alpha$ for $\alpha>0$ are given by $X_\alpha=\Fix(T)\oplus\Ran(I-T)^\alpha$. Thus it is possible at least in principle to verify whether a given initial vector leads to convergence of a particular polynomial rate. For characterisations of the spaces $\Ran(I-T)^\alpha$ for $0<\alpha\le1$, see \cite[Section~2]{DeLi01}, \cite[Theorem~6.1]{HaTo10} and  \cite{LinSine83}. In particular, if $X$ is reflexive then $x\in \Ran(I-T)^\alpha$ if and only if
$$\sup_{n\ge1}\bigg\|\sum_{k=1}^n k^{-(1-\alpha)} T^kx\bigg\|<\infty.$$
Note also that by \cite[Theorem~2.23]{DeLi01} we have that $\bigcup_{0<\alpha<1}\Ran(I-T)^\alpha$ is strictly contained in the closure of $\Ran(I-T)$ whenever the latter is not closed.
\item In showing that the convergence in \eqref{conv_gen} is arbitrarily slow when $\Ran(I-T)$ is not closed,  \cite[Theorem~1]{Mue88} in fact yields the stronger statement that, given any $\varepsilon>0$, the vector $x\in Z$ in the above proof can be chosen so as to satisfy $\|x\|<(1+\varepsilon)\sup_{n\ge0} r_n$. See \cite{BGM11} for various other notions of arbitrarily slow convergence which could equally have been used.
\item The rate $n^{-\alpha}$ in \eqref{rate_gen} is optimal in the sense that if the convergence were any faster for all $x\in X_\alpha$, with $X_\alpha$ as in the above proof, then the convergence in \eqref{conv2} would necessarily be exponentially fast. This follows from  \cite[Theorem~2.1]{Se2} combined with the uniform boundedness principle and the moment inequality; see for instance \cite[Corollary 7.2]{Ha05}. It follows in particular that the spaces $X_\alpha$ are distinct for distinct values of $\alpha>0$ when $\Ran(I-T)$ is not closed; see also \cite[Proposition~2.2]{DeLi01}.
\end{enumerate}
\end{rem}

Note that any reflexive space, and particular any Hilbert space, cannot contain a copy of $c_0$. Hence combining Theorem~\ref{thm:gen} with Theorem~\ref{Stolz} immediately leads to the following.

\begin{cor}\label{Hilbert_cor}
Let $X$ be a Hilbert space and suppose that $T\in\B(X)$ is similar to an operator whose numerical range is contained in a Stolz domain. Then the conclusions of Theorem~\ref{thm:gen} hold.
\end{cor}

\section{Inner inclinations and rates of convergence}\label{sec:inclination}
 
 The aim in this section is to introduce various quantities describing the geometric relationship between subspaces of a Hilbert space and to understand the connections between these quantities. These quantities will then be used to study the rate of convergence in the method of alternating projections. Throughout this section, we let $X$ be a Hilbert space and we consider a collection of $N\ge2$ closed subspaces $M_1,\dotsc,M_N$ of $X$. In order to avoid having to distinguish cases, we exclude the uninteresting situation in which all of the subspaces coincide with $X$. Furthermore, we let $M=M_1\cap\dotsc\cap M_N$ and, for $1\le k\le N$, we write $P_k$ for the orthogonal projection onto $M_k$. We also write $P_M$ for the orthogonal projection onto $M$ and we let $T\in\B(X)$ be the operator $T=P_N\cdots P_1$.
 
 The first quantity we consider, which was introduced in \cite{BGM10, BGM11}, is the \emph{Friedrichs number} $c(M_1,\dotsc,M_N)$, 
$$
 \begin{aligned}
c(M_1, \dotsc, M_N) &=  \sup\bigg\{ \frac{1}{N-1}\sum_{j\neq k} \ps{m_j}{m_k} : m_k \in M_k\cap M^{\perp}\;\mbox{for}\\[-7pt]
 & \qquad\qquad\qquad \quad1\le k\le N,\;\|m_1\|^2 + \cdots + \|m_N\|^2 =1 \bigg\} .
\end{aligned}$$
The Friedrichs number can be interpreted as the cosine of the `angle' between the subspaces $M_1,\dotsc,M_N$. 
It is well-known from \cite{De92, KaWe88} that when $N=2$ the norm $\|T^n-P_M\|$  can be nicely expressed in terms of the  Friedrichs number as 
 $$\|T^n-P_M\|=c(M_1,M_2)^{2n-1},\quad n\ge1.$$
  In the case of more than two subspaces, it is possible to write down certain expressions for $\|T^n-P_M\|$, $n\ge0$, as for instance in \cite[Theorem~4.5]{XuZik}, but no exact formula in terms of the Friedrichs number is known. However, it is possible to obtain bounds on these norms in terms of closely related geometric quantities, which can then be used to obtain estimates in terms of the Friedrichs number.
 
 Following \cite{PuReZa13}, we define the \emph{inner inclination} of the subspaces $M_1,\dotsc, M_N$ as
$$\iota(M_1,\dotsc, M_N)=\min_{1\le n\le N}\inf_{x \in M_n\backslash M}\max_{1\le k\le N}\frac{\dist(x,M_k)}{\dist(x,M)}.$$
The inner inclination is closely related to the \emph{inclination}
$$\ell(M_1,\dotsc,M_N)=\inf_{x\not\in M}\max_{1\le k\le N}\frac{\dist(x,M_k)}{\dist(x,M)}$$
  of the subspaces $M_1,\dotsc, M_N$ which is studied in \cite{BGM11}.  Indeed, it was proved in \cite[Proposition~3.9]{BGM11} that 
 $$ \ell(M_1,\dotsc,M_N)\le(N-1)^{1/2}\big(1-c(M_1,\dotsc,M_N)\big)^{1/2} $$
 and that
 \begin{equation}
\label{ell_lb}
\ell(M_1,\dotsc,M_N)\ge\frac{(N-1)}{2N}\big(1-c(M_1,\dotsc,M_N)\big).
\end{equation}
It also follows from the definitions that 
 \begin{equation}
\label{iota}
\iota(M_1,\dotsc, M_N)\ge \ell(M_1,\dotsc, M_N),
\end{equation}
 and it follows from \cite{BGM11} and \cite{PuReZa13}  that $\iota(M_1,\dotsc, M_N)=0$ if and only if $\ell(M_1,\dotsc, M_N)=0$, which in turn is equivalent to $c(M_1,\dotsc,M_N) = 1$ (the `angle' is zero) and to $M_1^\perp+\dots+ M_N^\perp$ not being closed in $X$. When these equivalent conditions are satisfied, we shall say that the subspaces $M_1,\dotsc,M_N$ are \emph{aligned}.

Consider the quantity
$$\ell_2(M_1,\dotsc ,M_N) = \left( \inf_{x\notin M}\frac{\sum_{k=1}^N\dist(x,M_k)^2}{\dist(x,M)^2}\right)^{1/2},$$
which was considered in \cite{Kas11} and which we shall call the $\ell^2$-\emph{inclination} of the subspaces $M_1,\dotsc,M_N$.
Thus $\ell_2(M_1,\dotsc ,M_N) $ is the smallest constant $C\ge 0$ such that the inequality
$$ \sum_{k=1}^N\dist(x,M_k)^2 \ge C^2\dist(x,M)^2$$
holds for all vectors $x\in X$. 
It follows from \cite[Remark 3.25]{Kas11} that we have
\begin{equation}\label{eq:31}
\ell_2(M_1,\dotsc, M_N) = (N-1)^{1/2}\big(1-c(M_1,\dotsc, M_N)\big)^{1/2}
\end{equation}
In particular, $\ell_2(M_1,\dotsc, M_N)=0$ if and only if the subspaces $M_1,\dotsc, M_N$ are aligned. We define  the \emph{inner} $\ell^2$-\emph{inclination} of the subspaces $M_1, \dotsc, M_N$ by
$$\iota_2(M_1,\dotsc ,M_N) =  \left( \min_{1\le n\le N} \inf_{x\in M_n\setminus M}\frac{\sum_{k=1}^N\dist(x,M_k)^2}{\dist(x,M)^2}\right)^{1/2}.$$
Thus $\iota_2$ is the smallest  constant $C\ge 0$ such that the inequality
$$ \sum_{k=1}^N\dist(x,M_k)^2 \ge C^2\dist(x,M)^2$$
holds for all vectors $x\in M_1 \cup \dotsc \cup M_N$. It follows that
\begin{equation}\label{eq:33}
\iota_2(M_1,\dotsc ,M_N) \ge \ell_2(M_1,\dotsc ,M_N).
\end{equation}

\begin{rem}
Note that it would be possible to develop the notion of an (inner) $\ell^p$-inclination for any $p$ with $1\le p\le\infty$. The usual (inner) inclination would then correspond to taking  $p=\infty$. We restrict ourselves here to the choice $p=2$ since this seems the most natural in the Hilbert space setting and since the the $\ell^2$-inclination can be expressed exactly in terms of the Friedrichs number by means of \eqref{eq:31}; see also Remark~\ref{rem:inf_vs_2} below.
\end{rem}

We now establish an estimate for the error $\|T^nx - P_Mx\|$, where $x\in X$ and $n\ge0$, first in terms of the inner $\ell^2$-inclination and then, in Corollary~\ref{cor:nor}, in terms of the Friedrichs number of the subspaces $M_1,\dotsc M_N$. We begin by recalling a result proved in \cite[Lemma~4.2.]{BGM11}.

\begin{lem}\label{lemma-u_j}
Given $N\ge2$ closed subspaces $M_1,\dotsc, M_N$ of a Hilbert space $X$, let $P_k$ denote the orthogonal projection onto $M_k$, $1\le k\le N$. Furthermore, let $P_M$ denote the orthogonal projection onto  $M=M_1\cap\dotsc\cap M_N$ and let $T=P_N\cdots P_1$.
Fix $x\in X$, and set $u_0 = x - P_Mx$ and, for $1\le k\le N$, $u_k = P_k \cdots P_1x - P_Mx$. Then
$$\|u_{k-1} - u_k\|^2 \le \|x-P_Mx\|^2 - \|Tx - P_Mx\|^2,\quad 1\le k\le N.$$
\end{lem}

The proof of the next theorem is similar to a part of \cite[Theorem 4.1]{BGM11} using an improvement of \cite[Theorem 2.4]{PuReZa13}.

\begin{thm}\label{thm:estIncl}
Consider the setting of Lemma~\ref{lemma-u_j}. Then, given $x\in X$, 
$$ \|T^nx - P_Mx\| \le \left({1- \frac{3\iota_2^2}{N^3}}\right)^{n/2} \left\|x-P_Mx\right\| ,\quad n\ge0,$$
where $\iota_2=\iota_2(M_1,\dotsc ,M_N)$.
\end{thm}

\begin{proof}
We begin by establishing an estimate for $\|Tx - P_Mx\|$. Let $x\in M_1$. Then
$\dist(x,M_1) = 0$. For $2\le k \le N$, keeping in mind that $x= P_1x$ and using the notation of Lemma~\ref{lemma-u_j}, we have that
$$\dist(x,M_k)=\|x-P_k x\|\le \|P_1x - P_kP_{k-1}\cdots P_1x\|=\|u_1-u_k\| $$
and hence 
$$\dist(x,M_k)\le\sum_{j=2}^k\|u_{j-1}-u_j\|.$$
Using Lemma~\ref{lemma-u_j} and the Cauchy-Schwarz inequality we obtain
$$\dist(x,M_k)^2\le (k-1)^2\left( \|x-P_Mx\|^2 - \|Tx - P_Mx\|^2\right),$$
and therefore  
$$\sum_{k=1}^N \dist(x,M_k)^2\le\frac{N^3}{3} \left( \|x-P_Mx\|^2 - \|Tx - P_Mx\|^2\right).$$
Since by definition of the inner $\ell^2$-inclination
$$  \sum_{k=1}^N \dist(x,M_k)^2\ge \iota_2^2 \|x-P_Mx\|^2$$
for all $x\in M_1$, we obtain 
$$\iota_2^2 \|x-P_Mx\|^2 \le \frac{N^3}{3} \left( \|x-P_Mx\|^2 - \|Tx - P_Mx\|^2\right),$$
and hence
\begin{equation}\label{eq:forM1}
\left\|Tx-P_Mx\right\| \le \left({1- \frac{3\iota_2^2}{N^3}}\right)^{1/2} \left\|x-P_Mx\right\|
\end{equation}
for all $x\in M_1$. Now let $x$ be an arbitrary element of $X$. We apply \eqref{eq:forM1} with $x$ replaced by $y = P_1x$, which is an element of $M_1$.
We have $P_1y = y = P_1x$ and $P_My = P_MP_1x = P_Mx$. Hence
$$ \|Tx - P_Mx\| = \|P_N\cdots P_1y - P_My\| \le \left({1- \frac{3\iota_2^2}{N^3}}\right)^{1/2}  \left\|y-P_My\right\| .$$
Now observe that
$$ \left\|y-P_My\right\| = \left\|P_1x-P_1P_Mx\right\| \le \left\|x-P_Mx\right\|$$
and thus the desired inequality, with $n=1$, holds for arbitrary $x\in X$. For $n\ge 2$, notice that $TP_Mx = P_MTx = P_Mx$. Thus we can write $T^nx - P_Mx = T(T^{n-1}x) - P_M(T^{n-1}x)$. By iteration, we obtain 
$$ \|T^nx - P_Mx\| \le \left({1- \frac{3\iota_2^2}{N^3}}\right)^{n/2}  \left\|x-P_Mx\right\|$$ 
for $n\ge2$ and hence for all $n\ge0$, as required. 
\end{proof}

Theorem \ref{thm:estIncl} shows that the convergence in \eqref{conv} is exponentially fast whenever $\iota_2(M_1,\dotsc,M_N)>0$. By \cite[Theorem~4.1]{BGM11} exponentially fast convergence is equivalent to having  $c(M_1,\dotsc,M_N)<1$, and hence \eqref{eq:31} and \eqref{eq:33} together show that $\iota_2(M_1,\dotsc,M_N)=0$ if and only if the spaces $M_1,\dotsc,M_N$ are aligned. The following corollary is also another simple consequence of Theorem \ref{thm:estIncl} and the statements in \eqref{eq:31} and \eqref{eq:33}. 

\begin{cor}\label{cor:nor}
Consider the setting of Lemma~\ref{lemma-u_j}. Then, given $x\in X$, 
\begin{equation}\label{c_bound}
\|T^nx - P_Mx\| \le \bigg(1-\frac{3(N-1)}{N^3}(1-c)\bigg)^{n/2}\left\|x-P_Mx\right\|,\quad n\ge0,
\end{equation}
where $c=c(M_1,\dotsc,M_N)$.
\end{cor}

\begin{rem}\label{rem:inf_vs_2}
It follows from \cite[Theorem~2.4]{PuReZa13}  that, for all $x\in X$,
$$ \|T^nx - P_Mx\| \le \left(1-\frac{\iota^2}{(N-1)^2}\right)^{n/2}  \left\|x-P_Mx\right\|,\quad n\ge0,$$ 
where   $\iota=\iota(M_1,\dotsc ,M_N)$. Using \eqref{ell_lb} and \eqref{iota} it is possible to deduce a bound involving only the Friedrichs number. However, resulting estimate is always worse than \eqref{c_bound}.
\end{rem}
 
\section{Convergence in the method of alternating projections}\label{sec:proj} 

Suppose as in Section~\ref{sec:inclination} that  $M_1,\dotsc, M_N$ are  $N\ge2$ closed subspaces of a Hilbert space $X$ and, for $1\le k\le N$, write $P_k$ for the orthogonal projection onto $M_k$.  It is shown in \cite[Lemma~5.1]{Cro08} that the numerical range $W(T)$ of the product $T=P_N\cdots P_1$  satisfies $W(T)\subset\Omega_N$, where 
\begin{equation}\label{eq:omeagN}
\Omega_{N}=\big\{\lambda\in\CC:|\lambda-2^{-N}|\le1-2^{-N}\;\mbox{and}\;|\arg(1-\lambda)|\le\theta_N\big\}
\end{equation}
and the angle $\theta_N\in(0,\pi/2)$ is defined recursively by $\theta_1=0$ and
$$\theta_{n+1}=\tan^{-1}\left(\frac{2\tan\theta_n}{1+2^{-n}}+\bigg(\frac{1-2^{-n}}{1+2^{-n}}\bigg)^{1/2}\right),\quad 1\le n<N.$$
Proposition~\ref{Proj_Ritt} below shows that $W(T)$ is also contained in a certain Stolz domain whose opening angle is determined by the geometric relationship between the subspaces $M_1,\dotsc,M_N$. 

\begin{prp}\label{Proj_Ritt}
Given $N\ge2$ closed subspaces $M_1,\dotsc, M_N$ of a Hilbert space $X$ with $M_1\cap\dotsc\cap M_N\ne X$, let $P_k$ denote the orthogonal projection onto $M_k$, $1\le k\le N$. Furthermore let $T=P_N\cdots P_1$.  Then $W(T)\subset \Omega_N\cap S_{\theta_0}$, where $\Omega_N$ is defined by \eqref{eq:omeagN} and
$$\sin \theta_0=\left(1-\frac{3(N-1)}{N^3}(1-c)\right)^{1/2}$$
with $c=c(M_1,\dotsc,M_N)$.  In particular,  $W(T)\subset S_\theta$ for all sufficiently large $\theta\in[0,\pi/2)$.
\end{prp}

\begin{proof}
Let $Z$ denote the closure of $\Ran(I-T)$. Then $Z^\perp=\Fix(T^*)$. Since $T^*=P_1\cdots P_N$, a simple argument shows that $\Fix(T^*)=\Fix(T)=M$,  where $M=M_1\cap\dotsc\cap M_N$. Hence $X=M\oplus Z$. Let $S$ denote the restriction of $T$ to $Z$. Then $W(T)$ is the convex hull of $W(S)\cup\{1\}$. Now using the estimate for $\|S\|=\|T-P_M\|$ implied by Corollary \ref{cor:nor}  we obtain
 $W(S)\subset\{\lambda\in\CC:|\lambda|\le\sin\theta_0\}$. Hence $W(T)\subset\Omega_N\cap S_{\theta_0}$, as required. The final claim follows from the fact that $\Omega_N\cap S_{\theta_0}\subset S_\theta$ for all sufficiently large $\theta\in[0,\pi/2)$.
\end{proof}
\begin{rem}\label{Ritt_rem}

 \begin{enumerate}[(a)]\label{Stolz_rem}
 \item \label{Friedrichs}
The result shows that $W(T)$ is contained in a Stolz domain whose half-angle is determined by the Friedrichs number. For $N=2$, \cite[Proposition~1.5]{Kla14} shows that 
$$W(T)\subset\big\{\lambda\in\CC:|\arg(1-\lambda)|\le\tan^{-1}(c(4-c^2)^{-1/2})\big\},$$ where $c=c(M_1,M_2)$. This gives a slightly sharper estimate on the angle of the Stolz domain than Proposition~\ref{Proj_Ritt}. Note also that by using Theorem~\ref{thm:estIncl} instead of Corollary~\ref{cor:nor} in the above proof, it is possible to obtain a slightly sharper bound on the angle $\theta_0$ involving the inner $\ell^2$-inclination rather than the  Friedrichs number.
 \item Combining Proposition~\ref{Proj_Ritt} with Theorem~\ref{Stolz} shows that the operator $T=P_N\cdots P_1$ satisfies the unconditional Ritt condition and in particular is a Ritt operator. The latter assertion  was first proved in \cite[Proposition~2.2]{Co07} for  conditional expectations; see also \cite[Proposition~2.3]{CoCuLi14}. A more direct way of showing that $T$ is a Ritt operator is to note that by \cite[Theorem~4.20]{Sto32} we have
\begin{equation*}\label{res_bound}
\|R(\lambda,T)\|\le\frac{1}{\dist(\lambda,W(T))}
\end{equation*}
for all $\lambda\in\CC$  outside the closure of $W(T)$. Since $W(T)$ is contained in a Stolz domain, a simple estimate establishes \eqref{eq:Ritt}.  Note in particular that better knowledge of the location of $W(T)$  leads to sharper estimates on the constant $C$ appearing in \eqref{eq:Ritt}. See \cite{CaZa01} for various related results. 
\item Note that if $T$ is a convex combination of operators whose numerical ranges lie in Stolz domains then $W(T)$ also lies in a Stolz domain. Thus the last part of Proposition~\ref{Proj_Ritt} extends straightforwardly to operators $T$ which are convex combinations of finite products of orthogonal projections on a Hilbert space, operators which were studied in \cite{BaLy10}. In particular, such operators satisfy the unconditional Ritt condition by Theorem~\ref{Stolz}.
\end{enumerate}
\end{rem}

We now come to the main result of this paper, which is a restatement of parts of the dichotomy result \cite[Theorem~4.1]{BGM11} but with two important additions. Specifically, the result gives an improved estimate in terms of the Friedrichs number on the actual rate of  convergence in \eqref{conv} in the case where it is known to be exponentially fast, and it also shows that, even when the convergence in \eqref{conv} is arbitrarily slow, there exists, for each $\alpha>0$, a rich supply of initial vectors for which the error decays like $o(n^{-\alpha})$ as $n\to\infty$. By showing that the initial vectors leading to very slow decay in \eqref{conv} must therefore come from the complement of a certain dense subspace of $X$, the result  provides a partial answer to a question raised in \cite[Remark~6.5(2)]{DeHu10}, where it is conjectured that such initial vectors in general must be chosen from outside $M\oplus(M_1^\perp+\cdots +M_N^\perp)$, which itself is a dense subspace of $X$. The result moreover shows that there exists a further dense subset $X_\infty$ of $X$ such that for $x\in X_\infty$ the decay is  {super-polynomially fast} in the sense of Section~\ref{sec:inclination}. Finally, the result also strengthens the unquantified statement \eqref{conv} by relating it to the convergence of a certain series which is unconditionally convergent.

\begin{thm}\label{Proj}
Given $N\ge2$ closed subspaces $M_1,\dotsc, M_N$ of a Hilbert space $X$, let $P_k$ denote the orthogonal projection onto $M_k$, $1\le k\le N$. Furthermore, let $P_M$ denote the orthogonal projection onto  $M=M_1\cap\dotsc\cap M_N$, and assume that $M\ne X$. If the sequence $(x_n)$ is defined recursively by $x_0=x$ and $x_{n+1}=P_N\cdots P_1x_n$ for $n\ge0$, then the series $\sum_{n\ge0}(x_n-x_{n+1})$ converges unconditionally to $x-P_Mx$ for all $x\in X$. In particular,
\begin{equation}\label{conv2}
\lim_{n\to\infty}\|x_n-P_Mx\|=0
\end{equation}
for all $x\in X$. 
If the subspaces $M_1,\dotsc,M_N$ are not aligned, then the convergence is exponentially fast and in fact
\begin{equation}\label{rate}
\|x_n-P_Mx\|\le \left(1-\frac{3(N-1)}{N^3}(1-c)\right)^{n/2}\|x-P_Mx\|,\quad n\ge0,
\end{equation}
for all $x\in X$, where $c=c(M_1,\dotsc,M_N)$. On the other hand, if the subspaces  $M_1,\dotsc,M_N$ are aligned, then the convergence is arbitrarily slow  but for each $\alpha>0$ there  exists a dense subspace $X_\alpha$ of $X$ such that 
\begin{equation*}\label{rate1}
\|x_n-P_Mx\|=o(n^{-\alpha}),\quad n\to\infty,
\end{equation*}
for all $x\in X_\alpha$. Furthermore, there exists a dense subspace $X_\infty$ of $X$ such that for all $x\in X_\infty$ the convergence in \eqref{conv2} is super-polynomially fast.
\end{thm}

\begin{proof}
Let $T\in \B(X)$ be given by $T=P_N\cdots P_1$, so that $x_n=T^nx$ for all $n\ge0$.  By Proposition~\ref{Proj_Ritt}, $W(T)$ lies in a Stolz domain, and the proof of that result shows that $\Fix(T)= M$. Since $\Ran(I-T)$ is closed if and only if the subspaces are not aligned by \cite[Theorem~4.1]{BGM11}, the result follows immediately from Corollaries~\ref{Hilbert_cor} and~\ref{cor:nor}.
\end{proof}

\begin{rem} 
\begin{enumerate}[(a)]
\item The estimate in \eqref{rate} can be sharpened slightly to an estimate involving the inner $\ell^2$-inclination rather than the Friedrichs number by replacing the application of Corollary~\ref{cor:nor} in the above proof with an application of Theorem~\ref{thm:estIncl}; see also Remark~\ref*{Stolz_rem}\eqref{Friedrichs}.
\item As was noted in the introduction, exponentially fast convergence in \eqref{conv2} is equivalent to the subspace $M_1^\perp+\dots+ M_N^\perp$ being closed in $X$; see \cite[Theorem~6.4]{DeHu10} and also \cite[Theorem~4.1]{BGM11} for a number of further  equivalent conditions. Thus both $M_1^\perp+\dots+ M_N^\perp$ and $\Ran(I-T)$ are closed if and only if $c(M_1,\dotsc,M_N)<1$. In fact, the identity
\begin{equation}\label{product}
I-T=\sum_{k=1}^N(I-P_k)\prod_{j=1}^{k-1}P_j
\end{equation}
implies that $\Ran(I-T)\subset M_1^\perp+\dots+ M_N^\perp$, and in particular both spaces are dense in $M^\perp$. However, the inclusion is in general strict when $c(M_1,\dotsc,M_N)=1$. Indeed, suppose that $N=2$ and that $M_1\cap M_2=M_1^\perp\cap M_2^\perp=\{0\}$. Then the assumption that $(I-P_2)(M_1^\perp)\subset \Ran(I-T)$ would already imply, by \eqref{product}, that for each $x\in X$ there exists $y\in X$ such that 
$$(I-P_2)\big((I-P_1)x-P_1y)\big)=(I-P_1)y.$$
But then both sides must lie in $M_1^\perp\cap M_2^\perp=\{0\}$, and therefore $(I-P_1)x-P_1y\in M_2$. It follows that $M_1^\perp\subset M_1+M_2$, and hence $M_1+M_2=X$. In particular, $M_1+M_2$ is closed, and it follows from \cite{Deu85} that $c(M_1,M_2)<1$.  We do not know whether $c(M_1,\dotsc,M_N)=1$ implies that $M_1^\perp+\dots+ M_N^\perp\subset\Ran(I-T)^\alpha$ for some $\alpha\in(0,1)$ or more generally that $M_1^\perp+\dots+ M_N^\perp\subset\bigcup_{\alpha>0}\Ran(I-T)^\alpha$; see also Remark~\ref*{gen_rem}\eqref{rem:frac}. A positive answer to either of the above would imply a positive answer to the conjecture made in \cite[Remark~6.5(2)]{DeHu10}.
\end{enumerate}
\end{rem}

\begin{cor}
Given $N\ge2$ closed subspaces $M_1,\dotsc, M_N$ of a Hilbert space $X$, let $P_k$ denote the orthogonal projection onto $M_k$, $1\le k\le N$. Furthermore, let $P_M$ denote the orthogonal projection onto  $M=M_1\cap\dotsc\cap M_N$. If the sequence $(x_n)$ is defined recursively by $x_0=x$ and $x_{n+1}=P_N\cdots P_1x_n$ for $n\ge0$, then
\begin{equation}\label{conv3}
\lim_{n\to\infty}\|x_n-P_Mx\|=0
\end{equation}
for all $x\in X$, and there exists a dense subspace $X_\infty$ of $X$ such that for all $x\in X_\infty$ the convergence in \eqref{conv3} is super-polynomially fast.
\end{cor}

%\bibliography{Projections}

\begin{thebibliography}{10}

\bibitem{Arv72}
W.~Arveson.
\newblock Subalgebras of {$C^{\ast} $}-algebras. {II}.
\newblock {\em Acta Math.}, 128(3-4):271--308, 1972.

\bibitem{BaBe14}
C.~Badea and B.~Beckermann.
\newblock Spectral sets.
\newblock In L.~Hogben, editor, {\em Handbook of linear algebra}, chapter~37.
  CRC Press, Boca Raton, FL, second edition, 2014.

\bibitem{BGM10}
C.~Badea, S.~Grivaux, and V.~M{\"u}ller.
\newblock A generalization of the {F}riedrichs angle and the method of
  alternating projections.
\newblock {\em C. R. Math. Acad. Sci. Paris}, 348(1-2):53--56, 2010.

\bibitem{BGM11}
C.~Badea, S.~Grivaux, and V.~M\"uller.
\newblock The rate of convergence in the method of alternating projections.
\newblock {\em Algebra i Analiz (St. Petersburg Math. J.)}, 23(3):1--30, 2011.

\bibitem{BaLy10}
C.~Badea and Y.I. Lyubich.
\newblock {Geometric, spectral and asymptotic properties of averaged products
  of projections in Banach spaces}.
\newblock {\em Studia Math.}, 201(1):21--35, 2010.

\bibitem{BaBoLe97}
H.H. Bauschke, J.M. Borwein, and A.S. Lewis.
\newblock The method of cyclic projections for closed convex sets in {H}ilbert
  space.
\newblock In {\em Recent developments in optimization theory and nonlinear
  analysis ({J}erusalem, 1995)}, volume 204 of {\em Contemp. Math.} Amer. Math.
  Soc., Providence, RI, 1997.

\bibitem{BaDeHu09}
H.H. Bauschke, F.~Deutsch, and H.~Hundal.
\newblock Characterizing arbitrarily slow convergence in the method of
  alternating projections.
\newblock {\em Int. Trans. Oper. Res.}, 16(4):413--425, 2009.

\bibitem{BesPel58}
C.~Bessaga and A.~Pe{\l}czy{\'n}ski.
\newblock On bases and unconditional convergence of series in {B}anach spaces.
\newblock {\em Studia Math.}, 17:151--164, 1958.

\bibitem{CaZa01}
V.~Cachia and V.A. Zagrebnov.
\newblock Operator-norm approximation of semigroups by quasi-sectorial
  contractions.
\newblock {\em J. Funct. Anal.}, 180(1):176--194, 2001.

\bibitem{Co07}
G.~Cohen.
\newblock Iterates of a product of conditional expectation operators.
\newblock {\em J. Funct. Anal.}, 242(2):658--668, 2007.

\bibitem{CoCuLi14}
G.~Cohen, C.~Cuny, and M.~Lin.
\newblock {Almost everywhere convergence of powers of some positive $L_p$
  contractions}.
\newblock {\em J. Math. Anal. Appl.}, 420(2):1129--1153, 2014.

\bibitem{CouSaCo90}
T.~Coulhon and L.~Saloff-Coste.
\newblock Puissances d'un op\'erateur r\'egularisant.
\newblock {\em Ann. Inst. H. Poincar\'e Probab. Statist.}, 26(3):419--436,
  1990.

\bibitem{Cro07}
M.~Crouzeix.
\newblock {Numerical range and functional calculus in Hilbert space}.
\newblock {\em J. Funct. Anal.}, 244(2):668--690, 2007.

\bibitem{Cro08}
M.~Crouzeix.
\newblock A functional calculus based on the numerical range: applications.
\newblock {\em Linear Multilinear Algebra}, 56:81--103, 2008.

\bibitem{CuLi15}
C.~Cuny and M.~Lin.
\newblock {Limit theorems for Markov chains by the symmetrization method}.
\newblock {\em J. Math. Anal. Appl.}, 2015, in press.

\bibitem{deLau98}
R.~deLaubenfels.
\newblock Similarity to a contraction, for power-bounded operators with finite
  peripheral spectrum.
\newblock {\em Trans. Amer. Math. Soc.}, 350(8):3169--3191, 1998.

\bibitem{DeDe99}
B.~Delyon and F.~Delyon.
\newblock {Generalization of von Neumann's spectral sets and integral
  representation of operators}.
\newblock {\em Bull. Soc. Math. France}, 127(1):25--41, 1999.

\bibitem{DeLi01}
Y.~Derriennic and M.~Lin.
\newblock Fractional {P}oisson equations and ergodic theorems for fractional
  coboundaries.
\newblock {\em Israel J. Math.}, 123:93--130, 2001.

\bibitem{Deu85}
F.~Deutsch.
\newblock Rate of convergence of the method of alternating projections.
\newblock In {\em Parametric optimization and approximation ({O}berwolfach,
  1983)}, volume~72 of {\em Internat. Schriftenreihe Numer. Math.}, pages
  96--107. Birkh\"auser, Basel, 1985.

\bibitem{De92}
F.~Deutsch.
\newblock The method of alternating orthogonal projections.
\newblock In {\em Approximation theory, spline functions and applications
  ({M}aratea, 1991)}, volume 356 of {\em NATO Adv. Sci. Inst. Ser. C Math.
  Phys. Sci.}, pages 105--121. Kluwer Acad. Publ., Dordrecht, 1992.

\bibitem{De01}
F.~Deutsch.
\newblock {\em Best approximation in inner product spaces}.
\newblock CMS Books in Mathematics. Springer, New York, 2001.

\bibitem{DeHu10}
F.~Deutsch and H.~Hundal.
\newblock Slow convergence of sequences of linear operators {II}: arbitrarily
  slow convergence.
\newblock {\em J.\ Approx.\ Theory}, 162(9):1717--1738, 2010.

\bibitem{DeHuSurvey1}
F.~Deutsch and H.~Hundal.
\newblock Arbitrarily slow convergence of sequences of linear operators: a
  survey.
\newblock In {\em Fixed-point algorithms for inverse problems in science and
  engineering}, volume~49 of {\em Springer Optim. Appl.}, pages 213--242.
  Springer, New York, 2011.

\bibitem{DeHu15}
F.~Deutsch and H.~Hundal.
\newblock Arbitarily slow convergence of sequences of linear operators.
\newblock In {\em Infinite products of operators and their applications},
  volume 636 of {\em Contemp. Math.}, pages 93--120. Amer. Math. Soc.,
  Providence, RI, 2015.

\bibitem{Est84}
J.~Esterle.
\newblock Mittag-{L}effler methods in the theory of {B}anach algebras and a new
  approach to {M}ichael's problem.
\newblock In {\em Proceedings of the conference on {B}anach algebras and
  several complex variables ({N}ew {H}aven, {C}onn., 1983)}, volume~32 of {\em
  Contemp. Math.}, pages 107--129. Amer. Math. Soc., Providence, RI, 1984.

\bibitem{FraMcI98}
E.~Franks and A.~McIntosh.
\newblock Discrete quadratic estimates and holomorphic functional calculi in
  {B}anach spaces.
\newblock {\em Bull. Austral. Math. Soc.}, 58(2):271--290, 1998.

\bibitem{GuDu97}
K.E. Gustafson and D.K.M. Duggirala.
\newblock {\em Numerical range}.
\newblock Springer, New York, 1997.

\bibitem{Ha05}
M.~Haase.
\newblock A functional calculus description of real interpolation spaces for
  sectorial operators.
\newblock {\em Studia Math.}, 171(2):177--195, 2005.

\bibitem{HaTo10}
M.~Haase and Y.~Tomilov.
\newblock Domain characterizations of certain functions of power-bounded
  operators.
\newblock {\em Studia Math.}, 196(3):265--288, 2010.

\bibitem{Hal62}
I.~Halperin.
\newblock The product of projection operators.
\newblock {\em Acta Sci. Math. (Szeged)}, 23:96--99, 1962.

\bibitem{KaPo08}
N.J. Kalton and P.~Portal.
\newblock Remarks on $\ell_1$ and $\ell_\infty$-maximal regularity for
  power-bounded operators.
\newblock {\em J. Aust. Math. Soc.}, 84(3):345--365, 2008.

\bibitem{Kas11}
M.~Kassabov.
\newblock Subspace arrangements and property {T}.
\newblock {\em Groups Geom. Dyn.}, 5(2):445--477, 2011.

\bibitem{KaWe88}
S.~Kayalar and H.L. Weinert.
\newblock Error bounds for the method of alternating projections.
\newblock {\em Math. Control Signals Systems}, 1(1):43--59, 1988.

\bibitem{Kla14}
H.~Klaja.
\newblock The numerical range and the spectrum of a product of two orthogonal
  projections.
\newblock {\em J. Math. Anal. Appl.}, 411(1):177--195, 2014.

\bibitem{Kre85}
U.~Krengel.
\newblock {\em Ergodic Theorems}.
\newblock Walter de Gruyter, Berlin, 1985.

\bibitem{KuWe04}
P.C. Kunstmann and L.~Weis.
\newblock {Maximal $L_p$-regularity for parabolic equations, Fourier multiplier
  theorems and $H^\infty$-functional calculus}.
\newblock In M.~Iannelli, R.~Nagel, and S.~Piazzera, editors, {\em Functional
  analytic methods for evolution equations}, volume 1855 of {\em Lecture Notes
  in Mathematics}, pages 65--312. Springer, Berlin, 2004.

\bibitem{LaLeMe13}
F.~Lancien and C.~Le Merdy.
\newblock {On functional calculus properties of Ritt operators}.
\newblock {\em Proc. Roy. Soc. Edinburgh}, 145(6):1239--1250, 2015.
%to appear, arXiv:1301.4875.

\bibitem{LinSine83}
M.~Lin and R.~Sine.
\newblock Ergodic theory and the functional equation {$(I-T)x=y$}.
\newblock {\em J. Operator Theory}, 10(1):153--166, 1983.

\bibitem{Ly99}
Yu. Lyubich.
\newblock Spectral localization, power boundedness and invariant subspaces
  under {R}itt's type condition.
\newblock {\em Studia Math.}, 134(2):153--167, 1999.

\bibitem{LeMe14}
C.~Le Merdy.
\newblock {$H^\infty$} functional calculus and square function estimates for
  {R}itt operators.
\newblock {\em Rev. Mat. Iberoam.}, 30(4):1149--1190, 2014.

\bibitem{Mue88}
V.~M\"uller.
\newblock {Local spectral radius formula for operators in Banach spaces}.
\newblock {\em Czechoslovak Math. J.}, 38(4):726--729, 1988.

\bibitem{NaZe99}
B.~Nagy and J.~Zem\'anek.
\newblock A resolvent condition implying power boundedness.
\newblock {\em Studia Math.}, 134(2):143--151, 1999.

\bibitem{NeSo06}
A.~Netyanun and D.C. Solmon.
\newblock Iterated products of projections in {H}ilbert space.
\newblock {\em Amer. Math. Monthly}, 113(7):644--648, 2006.

\bibitem{Pau84}
V.I. Paulsen.
\newblock Completely bounded homomorphisms of operator algebras.
\newblock {\em Proc. Amer. Math. Soc.}, 92(2):225--228, 1984.

\bibitem{Pau02}
V.I. Paulsen.
\newblock {\em Completely bounded maps and operator algebras}, volume~78 of
  {\em Cambridge Studies in Advanced Mathematics}.
\newblock Cambridge University Press, Cambridge, 2002.

\bibitem{PuReZa13}
E.~Pustylnik, S.~Reich, and A.J. Zaslavski.
\newblock Inner inclination of subspaces and infinite products of orthogonal
  projections.
\newblock {\em J. Nonlinear Convex Anal.}, 14(3):423--436, 2013.

\bibitem{ritt}
R.K. Ritt.
\newblock A condition that {$\lim_{n\to\infty}n^{-1}T^n=0$}.
\newblock {\em Proc. Amer. Math. Soc.}, 4:898--899, 1953.

\bibitem{Se2}
D.~Seifert.
\newblock {Rates of decay in the classical Katznelson-Tzafriri theorem}.
\newblock {\em J. Anal. Math.}, to appear.

\bibitem{Sto32}
M.H. Stone.
\newblock {\em Linear transformations in Hilbert space}, volume~15 of {\em
  American Mathematical Society Colloquium Publications}.
\newblock American Mathematical Society, Providence, RI, reprint of the 1932
  original edition, 1990.

\bibitem{vN49}
J.~von Neumann.
\newblock On rings of operators. {R}eduction theory.
\newblock {\em Ann. of Math.}, 50(2):401--485, 1949.

\bibitem{vNeu51}
J.~von Neumann.
\newblock Eine {S}pektraltheorie f\"ur allgemeine {O}peratoren eines unit\"aren
  {R}aumes.
\newblock {\em Math. Nachr.}, 4:258--281, 1951.

\bibitem{XuZik}
J.~Xu and L.~Zikatanov.
\newblock The method of alternating projections and the method of subspace
  corrections in {H}ilbert space.
\newblock {\em J. Amer. Math. Soc.}, 15(3):573--597, 2002.

\end{thebibliography}
%\bibliographystyle{plain}

\end{document}